\documentclass{amsart}
\usepackage{amssymb}
\usepackage[utf8]{inputenc}
\usepackage[english]{babel}
\usepackage{latexsym,cite,mathrsfs}
\usepackage{amsfonts}
\usepackage{amssymb, amsthm}
\usepackage{xcolor}

\newtheorem{theorem}{Theorem}[section]
\newtheorem{lemma}[theorem]{Lemma}
\newtheorem{proposition}{Proposition}
\theoremstyle{definition}

\theoremstyle{remark}

\numberwithin{equation}{section}

\begin{document}

\title[On degenerate $(q,p)$-Laplace equations]{On degenerate $(q,p)$-Laplace equations  corresponding 
	to an inverse spectral problem }


\author{Y.Sh.~Il'yasov}
\address{Institute of Mathematics of UFRC RAS, 112, Chernyshevsky str., 450008 Ufa, Russia}
\email{ilyasov02@gmail.com}

\author{N.F.~Valeev}
\address{Institute of Mathematics of UFRC RAS, 112, Chernyshevsky str., 450008 Ufa, Russia}
\email{valeevnf@mail.ru}

\subjclass[2020]{Primary }

\date{}

\dedicatory{}

\commby{}

\begin{abstract}

	Two main results are presented: 1)   a new class of applied problems that lead to equations with $(p,q)$-Laplace is presented; 2)   a  method for solving nonlinear boundary value problems involving $(p,q)$-Laplace with measurable unbounded coefficients is introduced. In the main result, the existence, uniqueness, and stability of the nonnegative weak solution to the equations of the form 
	$$
	-{\rm div}(\rho|\nabla u|^{q-2} \nabla u)-{\rm div}(|\nabla u|^{p-2}\nabla u)=\lambda b |u|^{q-2}u,~~p>q
	$$
	are proven. Additionally, an explicit formula that expresses solution of the equation through inverse optimal solution of the spectral problem 
	$$-{\rm div}(\rho|\nabla \phi|^{q-2}\nabla \phi)=\lambda b|\phi|^{q-2}\phi$$ 
	is presented. The advantage of the method is that the inverse optimal problem has a visible geometry and a simple variational structure, which makes it easy to solve it and, as a consequence,  find a solution to the associated nonlinear boundary value problem.

\end{abstract}

\maketitle

\section{Introduction}
The paper deals with the existence, uniqueness and stability of nonnegative weak solutions for the  following  nonlinear boundary value problem
\begin{equation} \label{eq:Nonl}
	\begin{cases}
		-{\rm div}(\rho |\nabla u|^{q-2}\nabla u)-{\rm div}(|\nabla u|^{p-2}\nabla u)=\lambda b |u|^{q-2}u,~~~ x\in \Omega,		\\ 
		~~u\bigr{|}_{\partial \Omega}=0.
	\end{cases}
\end{equation}
Hereafter, $\lambda \in \mathbb{R}$, $p=\frac{q \alpha}{\alpha-1}$, $q \geq 2, \alpha >  1$,  $\Omega$ is a bounded domain in $\mathbb{R}^N$, $N\geq 1$. We assume that   
\begin{align}
	\rho& \in L^1_{loc}(\Omega),~\rho>0,~\rho^{-s} \in L^1(\Omega)~~\mbox{with}~~s \in (\frac{N}{q}, +\infty)\cap [\frac{1}{q-1}, +\infty), \label{EqC}\\
	b& \in L^{\frac{\gamma}{\gamma-q}}, ~~meas\{x \in \Omega\mid~b(x) > 0\} > 0. \label{EqC2}
\end{align}
with some $\gamma$ satisfying $q < \gamma < q^*_s$, where 
$ \displaystyle{q^*_s=Nq s/(N(s+1)-q s)}$  if  $1\leq qs < N(s + 1)$, and $q^*_s$ is arbitrary, $1\leq q^*_s<+\infty$ if  $qs \geq N(s + 1)$. In what follows, we call $\rho$ the\textit{ density}. 


We denote by  $W^{1,q, \rho}:=W^{1,q}(\rho, \Omega)$ the weighted Sobolev space which is defined  as the set all functions $u \in L^q(\Omega)$ with finite norm
$$
\|u\|_{1,q, \rho}=\left(\int |u|^q \, dx +\int \rho |\nabla u|^q \, dx \right)^{1/q}<+\infty.
$$
$W^{1,q, \rho}_0:=W^{1,q}_0(\rho, \Omega)$ is a
closure of $C^\infty_0(\Omega)$ in $W^{1,q}(\rho, \Omega)$ with respect to the norm $\|\cdot\|_{1,q, \rho}$.  Due to \eqref{EqC},
the spaces $W^{1,q, \rho}$, $W^{1,q, \rho}_0$ are well-defined uniformly convex Banach spaces (see  Dr\'abek-Kufner-Nicolosi \cite{drabek}). 
Moreover, by the weighted Friedrichs inequality   the norm
$$
\|u\|_{W^{1,q, \rho}_0}=\left(\int \rho |\nabla u|^q \, dx \right)^{1/q}
$$
is equivalent to the norm $\|\cdot\|_{1,q, \rho}$ (see \cite{drabek}). 
We call  $u \in W^{1,q, \rho}_0\cap  W^{1,p}\setminus 0$ the weak solution of \eqref{eq:Nonl} if 
$$
\int\left(\rho ( |\nabla u|^{q-2}\nabla u, \nabla h)+( |\nabla u|^{p-2}\nabla u, \nabla h) -\lambda b|u|^{q-2}u h \right)\,dx=0
$$
is fulfilled for any $h \in C^\infty_0(\Omega)$.

In addition to problem \eqref{eq:Nonl}, special attention will be paid to the following spectral problem
\begin{equation} \label{eq:S}
	\begin{cases}	
		\mathcal{L}_\rho(\phi):=&-{\rm div}(\rho|\nabla \phi|^{q-2}\nabla \phi)=\lambda b|\phi|^{q-2}\phi,~~~ x\in \Omega, \\
		\phi \bigr{|}_{\partial \Omega}=&0.
	\end{cases}
\end{equation}
We call $\lambda(\rho)$ the eigenvalue of  $\mathcal{L}_\rho$	if there exists	$\phi(\rho) \in W^{1,q, \rho}_0\setminus 0$  calling an eigenfunction corresponding to $\lambda(\rho)$ such that
$$
\int\rho ( |\nabla \phi(\rho)|^{q-2}\nabla \phi(\rho), \nabla h)\,dx-\lambda(\rho) \int b|\phi(\rho)|^{q-2}\phi(\rho) h \,dx=0
$$
holds true for any $h\in  C^\infty_0(\Omega)$.
The principal eigenvalue of $\mathcal{L}_\rho$ can be found via the Courant–Fischer–Weyl variational principle
\begin{equation}\label{lambda1}
	\lambda_1(\rho)=\inf_{\phi \in C^\infty_0\setminus 0}\frac{\int \rho|\nabla \phi |^q dx}{\int b|\phi|^q\,dx},~~\rho \in L^1_{loc}(\Omega).
\end{equation}
Clearly, $\lambda_1(\rho)<+\infty$, $\forall \rho \in L^1_{loc}(\Omega)$, $\lambda_1(\rho)\geq 0$ if
$\rho(x)\geq 0$ a.e. in $\Omega$, moreover $\lambda_1(\rho)>0$ if
\eqref{EqC} is satisfied (see, e.g., \cite{drabek}). It is easy to check that $\lambda_1(\rho)=-\infty$ if and only if there exists a subset $\Omega^-:=\{x\in \Omega\mid \rho(x)<0\} \subset \Omega$ such that $meas(\Omega^-)>0$ and $\rho<0$  in $\Omega^-$.

Note that in the case $q=2$, if $\rho $ satisfies \eqref{EqC}, and (for simplicity) $b\equiv 1$, then by the First Representation Theorem (see \cite{Kato}, Theorem 2.1, p.322),  the quadratic form  $\int \rho \langle\nabla u,\nabla v \rangle\, dx$ generates a self-adjoint and semi-bounded operator $\mathcal{L}_\rho$ on $L^2(\Omega)$. Indeed, the applicability conditions of the theorem are satisfied since the closure of $C^\infty_0$ with respect to the norm $\|u\|_{1,2,\rho}:=(\int \rho |\nabla u|^2\, dx)^{1/2}$ generates the space $W_0^{1,2, \rho}$, while  $W_0^{1,2,\rho}$ is compactly embedded in $L^2(\Omega )$   (see, e.g., \cite{drabek},
p. 25).  Furthermore, the  spectrum of $\mathcal{L}_\rho$ consists of an infinite sequence of eigenvalues $\{\lambda_i(\rho) \}_{i=1}^{\infty}$,  each having finite multiplicity, and  $\inf_{i}\lambda_i(\rho) >0$ (see,  e.g., \cite{Kato}). In particular, it follows from this that the inverse problem of recovering the density $\rho(x)$ from a knowledge of the spectral data $\{\lambda_i(\rho) \}_{i=1}^{\infty}$ makes sense.

 The inverse problems is a classical and the most paramount problems in science and  applications (see, e.g.,  Chadan-Colton-P\"aiv\"arinta-Rundell \cite{chadan} and Gladwell \cite{glad}). In general, this work deals with the inverse problem of recovering the density function $\rho$ from a finite set of eigenvalues $\{\lambda_i \}_{i=1}^{m}$, $1\leq m<+\infty$. These type of problems, in a large, are ill-posed since having only finite spectral data, the inverse problem may possess infinitely many solutions (see, e.g., \cite{chadan, glad}). In \cite{ValIl_MatNot}, we proposed a new approach to such type problems, which makes their correct, and which is based on the conjecture that it is often possible to determine some a priory knowledge about the system. In particular,  for \eqref{eq:S}  one can assume that there exists  a priory function $\bar{\rho}$ which approximates the unknown density function $\rho_a$.
This naturally leads to the so-called \textit{inverse optimal problem} \cite{ValIl_JDE,ValIl_PhysD}: \textit{for a given $\bar{\rho}$ and a finite set of eigenvalues  $\{\lambda_i \}_{i=1}^{m}$, $1\leq m<+\infty$, find  $\hat{\rho}$ closest to $\bar{\rho}$, in  a certain norm, such that $\lambda_i(\hat{\rho})=\lambda_i$, $i=1,\ldots,m$}.  

In the present paper, we study the simplest version of this problem, which also makes sense for arbitrary $q\geq 2$:
\medskip

\par\noindent
$(P):$\,\,\textit{For a given $\lambda \in \mathbb{R}$ and $\bar{\rho} \in L^\alpha(\Omega)$, $\alpha>1$, $\bar{\rho}\geq 0$, find   a potential  $\hat{\rho} \in L^\alpha(\Omega)$ such that  $\lambda=\lambda_1(\hat{\rho})$ and }
\begin{equation}\label{Var}
	\|\bar{\rho}-\hat{\rho} \|_{L^\alpha}=\inf\{\|\bar{\rho}-\rho\|_{L^\alpha}:~~ \lambda=\lambda_1(\rho), ~\rho \geq 0,~\rho \in L^\alpha(\Omega)\}.
\end{equation}

\medskip

\noindent  Hereafter, we denote by 
$L^\alpha_s:=\{\rho \in L^\alpha:~\rho~~\mbox{satisfies}~~ \eqref{EqC}\}$ 
the subset equipped with the norm inherited from $L^\alpha$. Our  main results are as follows

\begin{theorem}\label{thm1}
	Assume that   $\bar{\rho} \in L^\alpha_s$ and $\lambda>\lambda_1(\bar{\rho})$. Then
	
	\begin{itemize}
		\item  there exists an unique  minimizer  $\hat{\rho} \in  L^\alpha_s$  of \eqref{Var}, moreover  $\lambda=\lambda_1(\hat{\rho})$, $\hat{\rho}\neq \bar{\rho}$,  and  $\hat{\rho}\geq \bar{\rho}>0$  in $\Omega$;
		\item   there exists a  nonnegative weak solution $\hat{u} \in W^{1,q, \rho}_0\cap  W^{1,p}\cap L^\infty$ of \eqref{eq:Nonl} with   $\rho=\bar{\rho}$. Furthermore, 
		\begin{align}\label{eq:MainEqual}
			&\hat{\rho}=\bar{\rho}+|\nabla \hat{u}|^{q/(q-1)}~~\mbox{a.e. in}~~ \Omega,\\
			&\lambda=\lambda_1(\hat{\rho}).
		\end{align}
	\end{itemize} 
\end{theorem}

\begin{theorem}\label{thm2}
	Assume that  $\bar{\rho} \in L^\alpha_s$. Then 
	
	$(1^o)$  \eqref{eq:Nonl} with $\rho=\bar{\rho}$  has no   nonnegative  weak solution if $\lambda\leq \lambda_1(\bar{\rho})$;

	$(2^o)$  \eqref{eq:Nonl}  with $\rho=\bar{\rho}$ and $\lambda>\lambda_1(\bar{\rho})$ has unique nonnegative weak solution in  $W^{1,q, \rho}_0\cap  W^{1,p}$.
\end{theorem}
As we proceed, we will use the notations $\hat{\rho}(\lambda, \bar{\rho})$, $\hat{u}(\lambda, \bar{\rho})$ for $\hat{\rho}$, $\hat{u}$, when a function's dependence on $\lambda$ and $\bar{\rho}$ needs to be considered. For $\lambda>0$, we define the subset
$$
L^s_{\alpha,\lambda}:=\{\rho \in L^\alpha_s \mid \lambda_1(\rho)<\lambda \}
$$
equipped with the norm inherited from $L^\alpha$. 

Our last result is about the continuity of solutions $\hat{\rho}(\lambda, \bar{\rho})$, $\hat{u}(\lambda, \bar{\rho})$ with respect to $\lambda$ and $\bar{\rho}$, which, in particular, means their stability.

\begin{theorem}\label{thm3}
	
	$(1^o)$ \, For any $\lambda >0$, 
		the maps $\hat{\rho}(\lambda, \cdot):  L^s_{\alpha,\lambda} \to L^\alpha_s$ and  $\hat{u}(\lambda, \cdot):  L^s_{\alpha,\lambda} \to  W^{1,\alpha
		}_0$ are continuous. Moreover,  $\|\hat{u}(\lambda, \cdot)\|_{1,q, (\cdot)}, \|\nabla\hat{u}(\lambda, \cdot)\|_{p} \in C(L^s_{\alpha,\lambda})$. 
		
		$(2^o)$\, 
		For any $\bar{\rho} \in L^\alpha_s$, 
		the maps $\hat{\rho}(\cdot,\bar{\rho} ):   [\lambda_1(\bar{\rho}),+\infty) \to L^\alpha_s$ and  $\hat{u}(\cdot, \bar{\rho}):  [\lambda_1(\bar{\rho}),+\infty) \to W^{1,q}_0$ are continuous.  Moreover,  $ \|\nabla\hat{u}(\cdot,\bar{\rho})\|_{p} \in C[\lambda_1(\bar{\rho}),+\infty)$
	
\end{theorem}

Proofs of the theorems are based on the development of the method introduced in \cite{ValIl_MatNot}. The feature of the method is that the main results are obtained within the framework of the variational problem \eqref{Var} (see  below its equivalent form \eqref{MinP}), that is  
the existence, uniqueness, and stability of the nonnegative weak solutions of nonlinear boundary value problem \eqref{eq:Nonl} follow from the corresponding results to variational problem  \eqref{Var}. It is important to note that in \eqref{Var}  the minimizer is sought in the  class of  Lebesgue spaces $L^\alpha$, which actually simplifies the problem, whereas this minimizer ultimately resides in the subset $L^\alpha_s$. As a result, the solution $\hat{u} \in W^{1,q, \rho}_0\cap  W^{1,p}\cap L^\infty$ to equation \eqref{eq:Nonl} is found, by and large, without applying the theory of Sobolev spaces, in particular, the Sobolev and Rellich–Kondrachov embedding theorems.



Equations of the type \eqref{eq:Nonl}  which involve bounded measurable coefficients have been being actively studied nowadays, see, e.g., survey of Papageorgiou \cite{Papageor}. We point out  works of  Colasuonno-Squassina \cite{Colasuonno}, Motreanu-Tanaka \cite{Motreanu},  Papageorgiou-Pudelko-Radulescu \cite{Papageorgiou},  Papageorgiou-Vetro-Vetro \cite{PapageorVV}, and Papageorgiou-Repovš-Vetro \cite{PapageorgiouVetro}  addresses the solvability of  equations involving  $(p,q)$ - Laplace with  bounded nonnegative coefficient  $\rho$, including the case when  $\rho(x)=0$ on some subset of $\Omega$.  However, we did not find any studies that dealt with equations involving $(p,q)$-Laplace with unbounded coefficients. Uniqueness of a positive solution for equations involving $(p,q)$-Laplace  with constant coefficient $\rho\equiv c$ and bounded function $b$ admitting sign-change has been obtained in work of Tanaka \cite{Tanaka}.  Notice  the constant coefficient $\rho\equiv c$ satisfies conditions \eqref{EqC}. To the best of our knowledge, there are no results on the uniqueness of a non-negative solution to equations involving $(p,q)$-Laplace with measurable unbounded coefficients. We are also not familiar with results obtained for equations with $(p,q)$-Laplace similar to those in Theorem \ref{thm3}.

Boundary value problems of type \eqref{eq:Nonl} involving $(p,q)$-Laplace  operators with measurable coefficients arise not only in the study of inverse problems.  The celebrated work of Zhikov \cite{Zikov} was the first to study such problems within the framework of nonlinear theory of elasticity to model anisotropic materials, in which the density coefficient $\rho(x)$ determines the geometry of a composite made from  different materials. We mention also that the $(p,q)$-Laplace operators with measurable coefficients can be  applied to describe steady-state solutions of reaction-diffusion systems in biophysics  and in chemical reaction design (see, e.g., Cherfils–Il'yasov \cite{Cherfils}).

This paper is organized as follows.
Section 2 contains some preliminaries. In Sections 3-5, we prove Theorems \ref{thm1}-\ref{thm3}. Appendix A contains an auxiliary lemma.

\section{Preliminaries}
In what follows, $\|\cdot\|_r$ denotes the norm in $ L^r(\Omega)$, $r \in [1,+\infty]$.   For a functional $f:L^\alpha \to \mathbb{R}$, $u, h \in L^\alpha$,  $\partial_\varepsilon^+f(u)(h)$,  $\displaystyle{\frac{d}{d \epsilon}} f(u)(h)$ denote the right derivative and derivative of $f(u+\varepsilon h)$ at $\varepsilon=0$, respectively.
Define
$$
R(\rho,v):=\frac{\int \rho|\nabla v |^q dx}{\int b |v|^q\,dx},~~~v \in C^\infty_0(\Omega),~\rho \in L^1_{loc}(\Omega).
$$

\begin{proposition}\label{prop1}
	$\lambda_1(\rho)$ is a concave functional in $L^1_{loc}$, that is 
	\begin{equation}\label{eq:concav}
		\lambda_1(t \rho_1+(1-t) \rho_2)\geq t\lambda_1(\rho_1)+(1-t)\lambda_1( \rho_2)~~\forall \rho_1,\rho_2 \in L^1_{loc},~~\forall t \in [0,1].
	\end{equation} 
\end{proposition}
\begin{proof} We show  \eqref{eq:concav} only for $\rho_1,\rho_2 \in L^1_{loc}$ such that $\lambda_1(\rho_1),  \lambda_1(\rho_2)>-\infty$, since for the rest  cases the proof is straightforward.  By \eqref{lambda1} we have
	\begin{align*}
		\lambda_1(t \rho_1+&(1-t) \rho_2)=\inf_{\phi \in C^\infty_0\setminus 0:\| b |\phi|^q \|_1=1}	\int(t \rho_1+(1-t) \rho_2)|\nabla \phi|^q\, dx \geq \\
		& t\inf_{\phi \in C^\infty_0\setminus 0:\| b |\phi|^q \|_1=1}\int\rho_1|\nabla \phi|^q\, dx +(1-t)\inf_{\phi \in C^\infty_0\setminus 0:\| b |\phi|^q \|_1=1} \int\rho_2|\nabla \phi|^q\, dx= \\ \\
		&t\lambda_1( \rho_1) +(1-t)\lambda(\rho_2),~~\forall t \in [0,1].
	\end{align*}
\end{proof}
\begin{proposition}\label{prop2}
	$\lambda_1(\rho)$ is a upper semicontinuous functional in $L^1_{loc}$.
\end{proposition}
\begin{proof} Let $(\phi_n)_{n=1}^\infty \subset C^\infty_0(\Omega)$ be a 
	countable dense set in $ C^\infty_0(\Omega)$. Clearly, 
	$$
	\lambda_1(\rho)=\inf_{n\ge 1}R(\rho,\phi_n), ~~\rho \in L^1_{loc}.
	$$
	Hence,	for any ${\displaystyle \tau \in \mathbb{R} }$, 
	$$
	\{\rho\in L^1_{loc}\mid \lambda_1(\rho)< \tau\}=\bigcup_{n=1}^{\infty}\{\rho\in L^1_{loc} \mid R(\rho,\phi_n)< \tau\}.
	$$
	It is easily seen that  $R(\cdot,\phi_n) \in C(L^1_{loc})$, $n=1,\ldots $. Hence the set $\{\rho\in L^1_{loc} \mid R(\rho,\phi_n)< \tau\}$ is open in $L^1_{loc}$, and therefore, $\{\rho\in L^1_{loc}\mid \lambda_1(\rho)< \tau\}$ is open for any ${\displaystyle \tau \in \mathbb{R} }$. This means that $\lambda_1(\rho)$ is  a upper semicontinuous functional.
	
\end{proof}
We shall employ the following result of Dr\'abek \cite{Drabek1}
\begin{lemma}\label{lem1}
	Assume that \eqref{EqC}, \eqref{EqC2} are fulfilled. Then 
	$\lambda_1(\rho)$ defined by \eqref{lambda1} is a  eigenvalue of $\mathcal{L}_{\rho}$, moreover the corresponding eigenfunction $\phi_1(\rho)$ belongs $W^{1,q, \rho}_0\cap L^\infty$ and  $\phi_1(\rho)\geq 0$ in $\Omega$.   Furthermore, there is no other eigenfunctions corresponding 	$\lambda_1(\rho)$ such that $\|\phi\|_{q}=1$.
\end{lemma}

Let us prove

\begin{lemma}\label{lem10} Assume that $\rho\in L^1_{loc}$,  $h\in C^\infty_0$, and $ h \geq 0$ in $\Omega$.  
	\begin{description}
		\item[(1)] If $\rho\geq 0$ in $\Omega$, then the map $[0,1)\ni \varepsilon \mapsto \lambda_1(\rho+\varepsilon h)$ admits a  right derivative at $\varepsilon=0$ such that $\partial_\varepsilon^+\lambda_1(\rho)(h)\geq 0$.  Moreover,   $\exists h_0\in C^\infty_0$, $h_0\geq 0$ such that $\partial_\varepsilon^+\lambda_1(\rho)(h_0)> 0$.
		\item[(2)] If $\rho$ satisfies \eqref{EqC}, and \eqref{EqC2} is fulfilled, then  the map $ \varepsilon \mapsto \lambda_1(\rho+\varepsilon h)$ is a differentiable function in a neighborhood of $0$, moreover
		\begin{equation}\label{eq:Val}
			\frac{d}{d\varepsilon}\lambda_1(\rho)=\frac{1}{\int b |\phi(\rho)|^q \, dx}\int |\nabla \phi_1(\rho)|^q h\, dx.
		\end{equation}
	\end{description}
\end{lemma}

\begin{proof} Let $h\in C^\infty_0$ such that $h \geq 0$ in $\Omega$. Note that if $p\geq 0$ in $\Omega$, then $\lambda_1(\rho)\geq 0$.  By Proposition \ref{prop1} the map $[0,1)\ni \varepsilon \mapsto \lambda_1(\rho+\varepsilon h)$ is concave function, and therefore, it is continuous and admits right derivative at $\varepsilon=0$.

	Observe
	\begin{align}
		\lambda_1(\rho+\varepsilon h) = &\inf_{\phi \in C^\infty_0\setminus 0}\left(\frac{\int \rho|\nabla \phi |^q dx}{\int b |\phi|^q\,dx}+\varepsilon\frac{\int h(x)|\nabla \phi |^q dx}{\int b |\phi|^q\,dx} \right)\geq \nonumber\\
		&\inf_{\phi \in C^\infty_0\setminus 0}\frac{\int \rho|\nabla \phi |^q dx}{\int b |\phi|^q\,dx}+\varepsilon\inf_{\phi \in C^\infty_0\setminus 0}\frac{\int h(x)|\nabla \phi |^q dx}{\int b |\phi|^q\,dx} = \lambda_1(\rho)+\varepsilon\lambda_1(h).\label{eq:lem2}
	\end{align}
	Hence, $\lambda_1(\rho+\varepsilon h)-\lambda_1(\rho)\geq  \varepsilon\lambda_1(h)$ which implies that  $\partial^+_\varepsilon\lambda_1(\rho)(h)\geq  0$. Clearly, there exists $h_0\in C^\infty_0$, $h_0>0$ such that $\lambda_1(h_0)>0$, and therefore by \eqref{eq:lem2} we have $\partial_\varepsilon^+\lambda_1(\rho)(h_0)\geq \lambda_1(h_0)> 0$.

	Assume  that \eqref{EqC} is satisfied. Then by Lemma \ref{lem1} there exists a unique nonnegative eigenfunction  $\phi_1(\rho) \in W^{1,q,\rho}_0$ of  $\mathcal{L}_{\rho}$ such that $\|\phi\|_{q}=1$. Since $\phi_1(\rho)$ is a minimizer of \eqref{lambda1},  we have
	\begin{align*}
		\lambda_1(p+\varepsilon h) =\inf_{\phi \in C^\infty_0\setminus 0}&\left(\frac{\int \rho|\nabla \phi |^q dx}{\int b |\phi|^q\,dx}+\varepsilon\frac{\int h(x)|\nabla \phi |^q dx}{\int b |\phi|^q\,dx} \right)\leq \\
		&\lambda_1(\rho)+ \varepsilon\frac{1}{\int b |\phi(\rho)|^q \, dx}\int |\nabla \phi_1(\rho)|^q h\, dx.
	\end{align*}
	 for sufficiently small $\varepsilon >0$.
	Thus, if $\varepsilon >0$ is sufficiently small, then  
	\begin{equation}\label{1ineq}
		\lambda_1(p+\varepsilon h)-\lambda_1(\rho)\leq  \varepsilon\frac{1}{\int b |\phi(\rho)|^q \, dx}\int |\nabla \phi_1(\rho)|^q h\, dx.
	\end{equation}
	On the other hand,
	\begin{align*}
		\lambda_1(p+\varepsilon h) =\inf_{\phi \in C^\infty_0\setminus 0} &\left(\frac{\int \rho|\nabla \phi |^q dx}{\int b |\phi|^q\,dx}+\varepsilon\frac{\int h(x)|\nabla \phi |^q dx}{\int b |\phi|^q\,dx} \right)\geq\\
		&\lambda_1(\rho)-\varepsilon\frac{\int h(x)|\nabla \phi_1(\rho) |^q dx}{\int b |\phi_1(\rho)|^q\,dx}, 
	\end{align*}
	for sufficiently small $\varepsilon >0$.
	Thus by \eqref{1ineq} we have 
	\begin{align*}
		-\varepsilon\frac{1}{\int b |\phi(\rho)|^q \, dx}\int |\nabla \phi_1(\rho)|^q h\, dx\leq \lambda_1(p+\varepsilon h)-\lambda_1(\rho)\leq \varepsilon\frac{1}{\int b |\phi(\rho)|^q \, dx}\int |\nabla \phi_1(\rho)|^q h\, dx,	 
	\end{align*}
	for sufficiently small $\varepsilon >0$, which implies  \eqref{eq:Val} and proof of \textbf{(2)}.
	
\end{proof}

\section{Proof of Theorem \ref{thm1}}

We prove $(1^o)$ and $(2^o)$ jointly. Let $\bar{\rho} \in L^\alpha_s(\Omega)$  and $\lambda>\lambda_1(\bar{\rho})$. Consider the constrained minimization problem
\begin{equation}\label{MinP}
	\hat{\theta}=\min\{\|\rho-\bar{\rho}\|_\alpha^\alpha:\rho \in M_{\lambda}\},
\end{equation}
where 
$M_{\lambda}:=\{\rho \in L^\alpha:~~\lambda\leq \lambda_1(\rho)\}$. 
Observe $M_{\lambda} \neq \emptyset$. Indeed, take $\rho \in L^\alpha$ such that $\lambda_1(\rho)>0$. Then  for any $a\geq\lambda$, we have $a=\lambda_1(\frac{a \cdot \rho}{\lambda_1(\rho)})$, $\lambda_1(\frac{a \cdot \rho}{\lambda_1(\rho)})\geq\lambda$, and thus,  $\frac{\lambda \cdot \rho}{\lambda_1(\rho)} \in M_{\lambda}$. 

Let $(\rho_n) \subset M_\lambda$ be a minimizing sequence of \eqref{MinP}, i.e., $\|\rho_n-\bar{\rho}\|_\alpha^\alpha \to \hat{\theta}$ as $n\to +\infty$. Then $(\|p_n\|^q_q)$ is bounded, therefore, the reflexivity
of $L^\alpha$ yields that there exists a subsequence, denoting again by $(\rho_n)$, such that
$\rho_n \rightharpoondown \hat{\rho}$ weakly in $L^\alpha$ for some $\hat{\rho} \in L^\alpha$. 

By Propositions \ref{prop1},  \ref{prop2}, the functional $\lambda_1(\rho)$ is concave and upper semicontinuous on $L^\alpha$, and therefore, it is weakly upper semicontinuous. Hence,   
$$
\lambda\leq \limsup_{n\to +\infty} \lambda_1(\rho_n)\leq \lambda_1(\hat{\rho}).
$$
Thus,  $\hat{\rho} \in
M_\lambda$, and consequently $\hat{\rho} \neq 0$. On the other hand, by the weak lower semi-continuity of $\|\rho-\bar{\rho}\|_\alpha^\alpha$ we have 
$$
\|\hat{\rho}-\bar{\rho}\|_\alpha^\alpha\leq \liminf_{n\to +\infty}\|\rho_n-\bar{\rho}\|_\alpha^\alpha=\hat{\theta}.
$$
Since $\hat{\rho} \in
M_\lambda$, this implies that $\|\hat{\rho}-\bar{\rho}\|_\alpha^\alpha=\hat{\theta}$, that is $\hat{\rho}$ is a minimizer of \eqref{MinP}. Recalling that $ \lambda_1(\hat{\rho})\geq \lambda > \lambda_1(\bar{\rho})$ we conclude that $\hat{\rho} \neq \bar{\rho}$. Taking into account that $\|(\cdot)-\bar{\rho}\|_\alpha^\alpha$ for $q>1$ is a convex functional and by Proposition \ref{prop1}, $M_\lambda$ is a convex set we infer that the minimizer $\hat{\rho}$ is unique.

The Kuhn–Tucker Theorem  (see, e.g.,   Zeidler \cite{ Zeidler},  Theorem 47.E, p. 394) yields that there exist $\mu_1,\mu_2 \in \mathbb{R}$, $|\mu_1|+|\mu_2|\neq 0$ such that the Lagrange function 
$$
\Lambda(\rho):=\mu_1 \|\rho-\bar{\rho}\|_\alpha^\alpha +\mu_2 (\lambda- \lambda_1(\rho)),
$$
satisfies the following conditions: 
\begin{align}
	&\Lambda(\hat{\rho})=\min_{\rho\in L^\alpha}\Lambda(\rho) \label{eq1}\\
	& \mu_2 (\lambda- \lambda_1(\hat{\rho}))=0,\label{eq2}\\
	&\mu_1, \mu_2 \geq 0.\label{eq3}
\end{align}
In view of that $\lambda_1(\hat{\rho})>0$, we have  $\hat{\rho}\geq 0$, and therefore, Lemma \ref{lem10} implies that
there exists $\partial_\varepsilon^+\lambda_1(\hat{\rho})(h) \geq 0$, $\forall h \in C^\infty_0$, ~$h \geq 0$. This by \eqref{eq1} yields
\begin{equation}\label{eq:ner}
	-\mu_1 \alpha \int(\bar{\rho}-\hat{\rho})|\bar{\rho}-\hat{\rho}|^{\alpha-2}h\, dx-\mu_2\partial_\varepsilon^+\lambda_1(\hat{\rho})(h)\geq 0, ~~\forall h \in C^\infty_0, ~h \geq 0.
\end{equation}
Suppose that $\mu_1=0$. Then $\mu_2 \neq 0$. By Lemma \ref{lem10}, there exists $h_0 \in C^\infty_0$ such that $\partial_\varepsilon^+\lambda_1(\hat{\rho})(h_0)> 0$. Thus by \eqref{eq:ner} we obtain
$$
-\mu_2\partial_\varepsilon^+\lambda_1(\hat{\rho})(h_0)\geq 0,
$$
which implies a contradiction since $-\mu_2 < 0$. In the same manner, it is derived that $\mu_2 \neq 0$, which by  \eqref{eq2} implies that $\lambda_1(\hat{\rho})=\lambda$.

Observe that \eqref{eq:ner} yields  
$$
-\mu_1 \alpha \int(\bar{\rho}-\hat{\rho})|\bar{\rho}-\hat{\rho}|^{\alpha-2}h\, dx\geq  0, ~~\forall h \in C^\infty_0, ~h \geq 0,
$$
and hence,
$$
\mu_1(\hat{\rho}-\bar{\rho})|\bar{\rho}-\hat{\rho}|^{\alpha-2}\geq 0~~\mbox{a.e. in}~~\Omega,
$$
which implies that $\hat{\rho}\geq \bar{\rho}>0$ a.e. in $\Omega$ and consequently 
$
0<\hat{\rho}^{-s}\leq \bar{\rho}^{-s}~~\Rightarrow~~ \hat{\rho}^{-s} \in L^1 $
with  $s \in (\frac{N}{q}, +\infty)\cap [\frac{1}{q-1}, +\infty)$. As a result,  $\hat{\rho} \in  L^\alpha_s$. 
Hence, Lemma \ref{lem10} yields that  $ \varepsilon \mapsto \Lambda(\hat{\rho}+\varepsilon h)$ is a differentiable function in a neighborhood of $0$. Thus,  \eqref{eq1} yields $\displaystyle{\frac{d}{d\varepsilon}\Lambda(\hat{\rho}+\varepsilon h)|_{\varepsilon=0}=0}$, and therefore, in virtue of  \eqref{eq:Val}, we have
$$
-\mu_1 \alpha \int (\bar{\rho}-\hat{\rho})|\bar{\rho}-\hat{\rho}|^{\alpha-2}h\, dx-\mu_2\frac{1}{\int b |\phi_1(\hat{\rho})|^q \, dx}\int |\nabla \phi_1(\hat{\rho})|^q h\, dx= 0,~~\forall h \in C^\infty_0,
$$
where by Lemma \ref{lem1}, $\phi_1(\hat{\rho}) \in W^{1,q, \hat{\rho}}_0\cap L^\infty$. Hence, 
\begin{equation}\label{eq:fin}
	\hat{\rho}=\bar{\rho}+\mu |\nabla \phi_1(\hat{\rho})|^{q/(\alpha-1)}~~\mbox{a.e. in}~~\Omega,
\end{equation}
where $\mu:=(\mu_2/(\alpha \mu_1\int b |\phi(\hat{\rho})|^q \, dx))^{1/(\alpha-1)}>0$. Substituting \eqref{eq:fin} into \eqref{eq:S} we obtain
$$
- {\rm div}(\bar{\rho}|\nabla \phi_1(\hat{\rho})|^{q-2} \nabla \phi_1(\hat{\rho}))-\mu {\rm div}(|\nabla \phi_1(\hat{\rho})|^{p-2} \nabla \phi_1(\hat{\rho}))=\lambda b |\phi_1(\hat{\rho})|^{q-2}\phi_1(\hat{\rho}),
$$
where $p=\frac{q \alpha}{\alpha-1}>q$. Note that $\mu |\nabla \phi_1(\hat{\rho})|^{q/(\alpha-1)}=\hat{\rho}-\bar{\rho} \in L^\alpha$, and therefore, $|\nabla \phi_1(\hat{\rho})| \in  L^p(\Omega)$. Furthermore, the inequality  $\hat{\rho}
\geq \bar{\rho}>0$ yields $ W^{1,q, \hat{\rho}}_0 \hookrightarrow  W^{1,q, \bar{\rho}}_0$. 
Hence,   $\hat{u}:=\mu^\frac{\alpha-1}{q}\phi_1(\hat{\rho})$ weakly satisfies  \eqref{eq:Nonl} and  $\hat{u} \in W^{1,q, \bar{\rho}}_0\cap L^\infty$, $|\nabla \hat{u} | \in  L^p(\Omega)$.  Since $\hat{u}\in L^\infty$ and $|\nabla \hat{u} | \in  L^p(\Omega)$, we have $\hat{u} \in W^{1,p}$. Equality  \eqref{eq:fin} implies that $\hat{\rho}=\bar{\rho}+|\nabla \hat{u}|^{q/(q-1)}$ a.e. in $\Omega$. 

\section{Proof of Theorem \ref{thm2}}

Suppose, contrary to our claim, that for $\lambda\leq \lambda_1(\bar{\rho})$, \eqref{eq:Nonl} with $\rho=\bar{\rho}$  has a nonnegative  weak solution $u \in W^{1,q, \rho}_0 $. Testing equation \eqref{eq:Nonl} with $u$ and dividing by $\int b |u|^q\, dx$, we obtain 
$$
\frac{\int \bar{\rho}|\nabla u |^q dx}{\int b |u|^q\,dx}+\frac{\int |\nabla u |^p dx}{\int b |u|^q\,dx}=\lambda.
$$
By  \eqref{lambda1} we have
$$
\lambda_1(\bar{\rho})+\frac{\int |\nabla u |^p dx}{\int b |u|^q\,dx}\leq \lambda.
$$
which implies a contradiction.


To prove $(2^o)$, we need the following
\begin{lemma}\label{lem:Nonl}
	Let $u \in W^{1,q, \rho}_0\cap  W^{1,p}$ be a nonnegative weak solution of \eqref{eq:Nonl}. Then  $\tilde{\rho}:=\bar{\rho}+|\nabla u|^{q/(q-1)}$ is a local minimum point of $\|\rho-\bar{\rho}\|^q_q$  on $M_{\lambda}$. 
\end{lemma}
\begin{proof}
	Introduce 
	\begin{equation}\label{eq:tilderho}
		\tilde{\rho}:=\bar{\rho}+|\nabla u|^{q/(\alpha-1)}. 
	\end{equation}
	Then 
	$\tilde{\rho}\geq \bar{\rho}>0$ a.e. in $\Omega$, and consequently, 
	$$
	0<\tilde{\rho}^{-s}\leq \bar{\rho}^{-s}~~\Rightarrow~~ \tilde{\rho}^{-s} \in L^1
	$$
	with  $s \in (\frac{N}{q}, +\infty)\cap [\frac{1}{q-1}, +\infty)$. Hence, $\tilde{\rho}$ satisfies \eqref{EqC}. Observe that $u$ is a nonnegative eigenvalue of $	\mathcal{L}_{\tilde{\rho}}$, that is
	$$
	\mathcal{L}_{\tilde{\rho}}(u):=- {\rm div}(\tilde{\rho}  |\nabla u|^{q-2}\nabla u)={\lambda} b|u|^{q-2}u.
	$$
	On the other hand, since $\tilde{\rho}$ satisfies \eqref{EqC}, Lemma \ref{lem1} implies that there exists an eigenfunction $\phi_1(\tilde{\rho}) \in W^{1,q, \tilde{\rho}}_0\cap L^\infty$ of $\mathcal{L}_{\tilde{\rho}}$,  $\phi_1(\tilde{\rho})\geq 0$ in $\Omega$, with  corresponding eigenvalue $\lambda_1(\tilde{\rho})$, i.e.,
	$$
	- {\rm div}(\tilde{\rho}  |\nabla \phi_1(\tilde{\rho})|^{q-2}\nabla \phi_1(\tilde{\rho}))=\lambda_1(\tilde{\rho}) b |\phi_1(\tilde{\rho})|^{q-2}\phi_1(\tilde{\rho}).
	$$
	From \eqref{lambda1} it follows that $\lambda \geq \lambda_1(\tilde{\rho})$. Hence,  Lemma \ref{lem:Picone} from Appendix implies that $\lambda = \lambda_1(\tilde{\rho})$. This and Lemma \ref{lem1} yields that $u=k \phi_1(\tilde{\rho})$ with some constant $k>0$.

	Since $\lambda =\lambda_1(\tilde{\rho})$,  $\tilde{\rho} \in \partial M_\lambda$.
	By Proposition \ref{prop1} the map $[0,1)\ni t \mapsto \lambda_1(t(\tilde{\rho}+h)+(1-t)\tilde{\rho})$ is concave function, and therefore
	$$
	\lambda_1(\tilde{\rho}+t h)-\lambda_1(\tilde{\rho}) \geq t(\lambda_1(\tilde{\rho}+ h)-\lambda_1(\tilde{\rho}) ) =0, ~~\forall h\in L^\alpha~~\mbox{such that}~~ \tilde{\rho}+h \in \partial M_{\lambda},
	$$
	and consequently, by Lemma  \ref{lem10} we have
	\begin{equation} \label{eq:lambgeq0}
		\int |\nabla u|^q\cdot h\, dx = \frac{1}{k}\lim_{t\to 0}\frac{\lambda_1(\tilde{\rho}+t h)-\lambda_1(\tilde{\rho})}{t}\geq 0, \forall h\in C^\infty_0 ,~~ \tilde{\rho}+h \in \partial M_{\lambda}
	\end{equation}
	Note that \eqref{eq:tilderho} implies 
	$$
	D_{\tilde{\rho}} \|\tilde{\rho}-\bar{\rho}\|^\alpha_\alpha(h)=\alpha \int (\bar{\rho}-\tilde{\rho})|\bar{\rho}-\tilde{\rho}|^{\alpha-2}h\, dx|_{\rho=\tilde{\rho}}=\alpha\int |\nabla u|^\alpha \, dx,~~\forall h\in L^\alpha.
	$$
	Hence, by \eqref{eq:lambgeq0},
	\begin{equation}\label{eqDp}
		D_{\tilde{\rho}} \|\tilde{\rho}-\bar{\rho}\|^\alpha_\alpha(h)\geq 0, ~~~  \forall h\in C^\infty_0~~\mbox{such that}~~ \tilde{\rho}+h \in \partial M_{\lambda}.
	\end{equation}
	By the Taylor expansion, 
	$$ 
	\|\tilde{\rho}-\bar{\rho}+h\|^\alpha_\alpha= \|\tilde{\rho}-\bar{\rho}\|^\alpha_\alpha+  D_{\tilde{\rho}} \|\tilde{\rho}-\bar{\rho}\|^\alpha_\alpha(h)+o(\|h\|_\alpha),
	$$
	for sufficiently  small  $\|h\|_{\alpha}$.
	This by \eqref{eqDp} implies 
	$$
	\|\tilde{\rho}-\bar{\rho}+h\|^\alpha_\alpha\geq \|\tilde{\rho}-\bar{\rho}\|^\alpha_\alpha,
	$$ 
	for any  $h\in L^\alpha$ such that $ \tilde{\rho}+h \in \partial M_{\lambda}$ and sufficient small $\|h\|_\alpha$. Thus, indeed,  $\tilde{\rho}$ is a local minimum point of $\|\rho-\bar{\rho}\|^\alpha_\alpha$  on $M_{\lambda}$.

\end{proof}

Let us conclude the proof of $(2^o)$.
By the above, \eqref{eq:Nonl} possess a weak nonnegative solution $\hat{u}$ such that the functional $\|\rho-\bar{\rho}\|^\alpha_\alpha$ admits a global minimum at $\hat{\rho}=\bar{\rho}+|\nabla \hat{u}|^{q/(\alpha-1)}$ on $M_{\lambda}$ and  $\lambda=\lambda(\hat{\rho})$. Assume that there exists a second weak solution $\bar{u}$ of \eqref{eq:Nonl}. Then by Lemma \ref{lem:Nonl}, $\tilde{\rho}=\bar{\rho}+|\nabla \bar{u}|^{q/(\alpha-1)}$ is a  local minimum point of $\|\rho-\bar{\rho}\|^\alpha_\alpha$  on $M_{\lambda}$. However, due to the strict convexity of $M_{\lambda}$ and $\|\rho-\bar{\rho}\|^\alpha_\alpha$ this is possible  only if 
$\tilde{\rho}=\hat{\rho}$, and consequently $|\nabla \hat{u}|=|\nabla \bar{u}|$. Hence,
by the above 
$ \mathcal{L}_{\hat{\rho}}(\bar{u})=\lambda |\bar{u}|^{q-2}\bar{u}$
and $\mathcal{L}_{\hat{\rho}}(\hat{u})=\lambda |\hat{u}|^{q-2}\hat{u}$ with $\lambda=\lambda(\hat{\rho})$.
Moreover, since $|\nabla \hat{u}|=|\nabla \bar{u}|$, we have $\|\bar{u}\|_{1,q}=\|\hat{u}\|_{1,q}$. Hence,  Lemma \ref{lem1} yields that $\bar{u}=\hat{u}$.

\section{Proof of Theorem \ref{thm3}}

$(1^o)$ \,Let $\lambda>0$, $\bar{\rho} \in L^s_{\alpha,\lambda}$, and thus,  $\lambda_1(\bar{\rho})<\lambda$. 
Suppose, contrary to our claim, that there exists a sequence $\bar{\rho}_n \in L^s_{\alpha,\lambda}$, $n=0,1,\ldots$  such that $\bar{\rho}_n \to \bar{\rho}$ in $ L^\alpha$ as $n\to \infty$, and $\inf_{n}\|\hat{\rho}(\lambda, \bar{\rho}_n)-\hat{\rho}(\lambda, \bar{\rho})\|_\alpha=\delta>0$. 	Denote $\hat{\rho}_n:=\hat{\rho}(\lambda, \bar{\rho}_n)$, $n=1,\ldots$, $\hat{\rho}_0:=\hat{\rho}(\lambda, \bar{\rho})$, $\bar{\rho}_0=\bar{\rho}$,  
\begin{equation}\label{MinPn}
	a_n:=	\|\bar{\rho}_n -\hat{\rho}_n\|_\alpha=\inf_{\rho \in M_{\lambda}}\|\bar{\rho}_n -\rho\|_\alpha,~~n=0,1,\ldots.
\end{equation}
Note that by \eqref{MinP}, $a_n=\hat{\theta}_n^{1/\alpha}$. 
Since $\bar{\rho}_n \to \bar{\rho}_0$ in $L^\alpha$, 
\begin{equation}\label{eq:theta}
	\hat{a}:=\lim_{n\to +\infty}a_n= \lim_{n\to +\infty}\inf_{\rho \in M_{\lambda}}\|\bar{\rho}_n -\rho\|_\alpha \leq \inf_{\rho \in M_{\lambda}}\|\bar{\rho}_0 -\rho\|_\alpha=a_0.
\end{equation}
Denote $\zeta_n:=\bar{\rho}_n -\hat{\rho}_n$,   $n=0,1,\ldots$. By convexity of the norm $\|\cdot\|_\alpha$ we have
\begin{equation}\label{eq:convex1}
	\|\frac{1}{2}( \zeta_n+ \zeta_0)\|_\alpha\leq \frac{1}{2} \|\zeta_n\|_\alpha+\frac{1}{2} \|\zeta_0\|_\alpha, 
\end{equation}
while the convexity of $M_\lambda$ implies that
$
\frac{1}{2}( \hat{\rho}_n+ \hat{\rho}_0) \in M_\lambda,~~n=0,1,\ldots.
$
Hence, 
$$
a_0\leq \|\frac{1}{2}( \hat{\rho}_n+ \hat{\rho}_0)-\bar{\rho}_0\|_\alpha=\|\frac{1}{2}( \zeta_n+ \zeta_0)\|_\alpha,~~n=0,1,\ldots , 
$$ 
and thus, taking into account \eqref{eq:convex1}, we obtain
\begin{equation}\label{eq:theta2}
	2a_0\leq \|\zeta_n+ \zeta_0\|_\alpha\leq a_n+ a_0,~~n=0,1,\ldots.
\end{equation}
By \eqref{eq:theta}, $\hat{a}\leq a_0$. If $\hat{a}<a_0$, then \eqref{eq:theta2} implies a contraction. Thus,  $a_n \to a_0$ as $n\to +\infty$, and 
\begin{equation}\label{eq:theta3}
	\|\zeta_n+ \zeta_0\|_\alpha \to 2a_0~~\mbox{as}~~n \to +\infty.
\end{equation}
By the above we have $\|\zeta_n\|_\alpha=a_n \to a_0$  as $n \to +\infty$. Hence, in view of that $L^\alpha$ is an uniformly convex space, \eqref{eq:theta3} yields    
$$
\|\hat{\rho}_n-\hat{\rho}_0+\bar{\rho}_n-\bar{\rho}_0 \|_\alpha=\|\zeta_n- \zeta_0\|_\alpha \to 0~~\mbox{as}~~n \to +\infty.
$$
Since 
$$
\|\hat{\rho}_n-\hat{\rho}_0+\bar{\rho}_n-\bar{\rho}_0 \|_\alpha\geq \|\hat{\rho}_n-\hat{\rho}_0\|_\alpha- \|\bar{\rho}_n-\bar{\rho}_0 \|_\alpha
$$ 
and $\|\bar{\rho}_n-\bar{\rho}_0 \|_\alpha\to 0$ as $n \to +\infty$, we obtain $\|\hat{\rho}_n-\hat{\rho}_0\|_\alpha \to 0$, which is  a contradiction. Thus, the map $\hat{\rho}(\lambda, \cdot):  L^\alpha_s \to L^\alpha_s$ is continuous.

Let us show   the continuity of the map  $\hat{u}(\lambda, \cdot)$. 
In view of that  $\hat{\rho}(\lambda, \cdot):  L^s_{\alpha,\lambda} \to L^\alpha_s$ is continuous, and 
$$
|\nabla \hat{u}(\lambda, \rho)|=(\hat{\rho}(\lambda, \rho)-\rho)^\frac{\alpha-1}{q}~~\mbox{a.e. in}~~ \Omega,~~\rho \in L^s_{\alpha,\lambda},
$$
we infer that  $|\nabla \hat{u}(\lambda, \cdot)|:L^s_{\alpha,\lambda} \to L^\frac{q \alpha}{\alpha-1} \equiv L^p\subset L^q$ is continuous, and consequently $\|\nabla\hat{u}(\lambda, \cdot)\|_{q}, \|\nabla\hat{u}(\lambda, \cdot)\|_{p} \in C(L^s_{\alpha,\lambda})$. Moreover, this implies that  
\begin{equation}\label{eq:rhorho}
\hat{\rho}(\lambda, \cdot) |\nabla\hat{u}(\lambda, \cdot)|^q:  L^s_{\alpha,\lambda} \to L^1~~\mbox{is continuous}.
\end{equation}
From this  and $\|\nabla\hat{u}(\lambda, \cdot)\|_{p} \in C(L^s_{\alpha,\lambda})$  it is not hard to show  using \eqref{eq:Nonl} that $\hat{u}(\lambda, \cdot):  L^s_{\alpha,\lambda} \to L^q$ is continuous. Hence, $\hat{u}(\lambda, \cdot) :  L^s_{\alpha,\lambda} \to  W^{1,q}_0$ is continuous, and by \eqref{eq:rhorho}, we infer $\|\hat{u}(\lambda, \cdot)\|_{1,q, (\cdot)} \in C(L^s_{\alpha,\lambda})$.

$(2^o)$ \,	Let $\bar{\rho} \in L^\alpha_s$, $\lambda \in (\lambda_1(\bar{\rho}), +\infty)$.	Suppose, contrary to our claim, that there exists a sequence $\lambda_n \in (\lambda_1(\bar{\rho}), +\infty)$, $n=0,1,\ldots$, $\lambda_n\to \lambda$ as $n\to \infty$  such that $\inf_{n}\|\hat{\rho}(\lambda_n, \bar{\rho})-\hat{\rho}(\lambda, \bar{\rho})\|_\alpha=\delta>0$. Denote $\hat{\rho}^n:=\hat{\rho}(\lambda_n, \bar{\rho})$, $n=1,\ldots$, $\hat{\rho}^0:=\hat{\rho}(\lambda, \bar{\rho})$, $\lambda_0=\lambda$.

Clearly, there exists a monotone subsequence, which we again denote by $\lambda_n$,  such that $\lambda_n \to \lambda$.  Assume, for instance, that $\lambda\geq \ldots \geq \lambda_n\geq \lambda_{n+1}$, $n=1,\ldots$. 
Then 
\begin{equation}\label{eq:Ml}
	M_{\lambda} \subset M_{\lambda_{n+1}} \subset M_{\lambda_{n}}, ~~n=1,\ldots
\end{equation} 
Denote
$$
a^n:=\|\hat{\rho}^n -\bar{\rho}\|_\alpha=\inf_{\rho \in M_{\lambda_n}}\|\bar{\rho} -\rho\|_\alpha,~n=0,\ldots.
$$
From \eqref{eq:Ml} it follows $a^0\leq \ldots \leq a^n \leq \ldots$, $n=1,\ldots$. 
Define $t_n=\lambda_{n}/\lambda$, $n=1,\ldots$. Then $t_n \to 1$ as $n\to +\infty$. Due to homogeneity of $\lambda(\rho)$,  $\lambda_n=\lambda(t_n\hat{\rho}^0)$, and therefore, $t_n\hat{\rho}^0 \in M_{\lambda_{n}}$, $\forall n\geq 1$. Hence, 
$a^0\leq a^n \leq \|t_n\hat{\rho}^0 -\bar{\rho}\|_\alpha$, $\forall n\geq 1$.
Since $\|t_n\hat{\rho}^0 -\bar{\rho}\|_\alpha \to a^0$ as $n\to +\infty$, 
this yields $\lim_{n\to \infty} a^n \to a^0$.

Note that $\|\hat{\rho}^n -\bar{\rho}+\hat{\rho}^0 -\bar{\rho}\|_\alpha \leq a^n+a^0$, $\forall n\geq 1$. 
Due to convexity of $M_{\lambda_n}$, $(\hat{\rho}^n +\hat{\rho}^0) /2 \in M_{\lambda_n}$, and therefore, $\|\hat{\rho}^n -\bar{\rho}+\hat{\rho}^0 -\bar{\rho}\|_\alpha=2\|(\hat{\rho}^n +\hat{\rho}^0) /2 -\bar{\rho}\|_\alpha\geq 2a^n$. Thus, we have
$$
2a^n\leq \|\hat{\rho}^n -\bar{\rho}+\hat{\rho}^0 -\bar{\rho}\|_\alpha \leq a^n+a^0,
$$
which implies that 
\begin{equation}
	\|\hat{\rho}^n -\bar{\rho}+\hat{\rho}^0 -\bar{\rho}\|_\alpha \to 2a^0~~\mbox{as}~~n\to +\infty.
\end{equation}
Because of $\lim_{n\to \infty}\|\hat{\rho}^n -\bar{\rho}\|_\alpha \to a^0$,  $\|\hat{\rho}^0 -\bar{\rho}\|_\alpha = a^0$, and $L^\alpha$ is an uniformly convex space, this yields 
$$
\lim_{n\to \infty}\|(\hat{\rho}^n -\bar{\rho})-(\hat{\rho}^0 -\bar{\rho})\|_\alpha=\lim_{n\to \infty}\|\hat{\rho}^n -\hat{\rho}^0\|_\alpha = 0.
$$
We get a contradiction, and therefore, we obtain our claim. 

Similar to the proof of $(1^o)$, it can be shown that the continuity of  $\hat{\rho}(\cdot,\bar{\rho} ):   [\lambda_1(\bar{\rho}),+\infty) \to L^\alpha_s$  implies that   $\hat{u}(\cdot, \bar{\rho}):  [\lambda_1(\bar{\rho}),+\infty) \to W^{1,q}_0$ is  continuous, and   $ \|\nabla\hat{u}(\cdot,\bar{\rho})\|_{p} \in C[\lambda_1(\bar{\rho}),+\infty)$.

\section{Appendix A}

In the next result, we will rely on work of Heinonen-Kipelainen-Martio \cite{Heinonen}, so we need  some preliminaries from it. In \cite{Heinonen},  weighted Sobolev spaces $H^{1,q}(\Omega, \rho)$, $H_0^{1,q}(\Omega, \rho)$  with so-called p-admissible weights $\rho$ are considered (see pages 7-8 in \cite{Heinonen}). It can be verified by applying the theory of $A_p$-weights from \cite{Heinonen} that the function $\rho $ satisfying \eqref{EqC} is a $q$-admissible weight. 
Furthermore, the weighted space $H_0^{1,q}(\Omega, \rho)$ in \cite{Heinonen} corresponds to the space  $W^{1,q, \rho}_0$  since their definitions are the same because of the Poincar\'e inequality. A function $v \in L^{q, \rho}$ is called the gradient of $u \in W^{1,q, \rho}_0$ and  denoted by $v = \nabla u$ if there exists a sequence $\psi_j \in C^\infty_0(\Omega)$ such that $\int\rho |\nabla \psi_j- v|^q \, dx \to 0$ as $j\to +\infty$ (cf. \cite{Heinonen}, page 12). 

\begin{lemma} Assume that $\rho$ satisfies \eqref{EqC}. 
	Let	$\phi \in W^{1,q, \rho}_0\cap L^\infty$, $u \in W^{1,q, \rho}_0$ and $u, \phi\geq 0$ a.e. in $\Omega$. Then $\displaystyle{\frac{\phi^q}{(u+\epsilon)^{q-1}} \in  W^{1,q, \rho}_0\cap L^\infty}$ and its  gradient is expressed as follows:
	\begin{equation}\label{eq:Picone}
		\nabla\left(\frac{\phi^q}{(u+\epsilon)^{q-1}}\right)=q \frac{\phi^{q-1}}{(u+\epsilon)^{q-1}}\nabla \phi-(q-1)\frac{\phi^q}{(u+\epsilon)^{q}}\nabla u.
	\end{equation}
\end{lemma}
\begin{proof}
	The proof is similar  in spirit to Lemma B.1 in Bobkov-Tanaka \cite{Bobkov2} so we only sketch it.
	Applying Theorems 1.18 and 1.24 from \cite{Heinonen} according to arguments in \cite{Bobkov2} we derive that $\phi^q, 1/(u+\epsilon)^{q-1}\in W^{1,q, \rho}_0\cap L^\infty$, and their gradients can be calculated according to the classical rules. Hence, Theorem 1.24 from \cite{Heinonen} implies that  $\phi^q/(u+\epsilon)^{q-1}\in W^{1,q, \rho}_0\cap L^\infty$  and its gradient is expressed by \eqref{eq:Picone}. 
\end{proof}
Let $\rho$  satisfies \eqref{EqC}, $\phi \in W^{1,q, \rho}_0\cap L^\infty$, $u \in W^{1,q, \rho}_0$, and  $u, \phi \geq 0$ a.e. in $\Omega$. For $\epsilon >0$, denote
\begin{align*}
	L(\phi,u+\epsilon)&= |\nabla \phi|^q+(q-1)\frac{\phi^q}{(u+\epsilon)^{q}}|\nabla u|^q - q\frac{\phi^{q-1}}{(u+\epsilon)^{q-1}}|\nabla u|^{q-2}(\nabla \phi,\nabla u),\\
	R(\phi,u+\epsilon) &=  |\nabla \phi|^q-|\nabla u|^{q-2}(\nabla\left(\frac{\phi^q}{(u+\epsilon)^{q-1}}\right),\nabla u).
\end{align*}

\begin{lemma} \label{allegH} Assume that $\rho$  satisfies \eqref{EqC}, $\phi \in W^{1,q, \rho}_0\cap L^\infty$, $u \in W^{1,q, \rho}_0$, and  $u, \phi\geq 0$ a.e. in $\Omega$. 	Then $\rho L(\phi,u+\epsilon)=\rho R(\phi,u+\epsilon)$, and $\int \rho L(\phi,u+\epsilon)\,dx \geq 0$. 
\end{lemma}
\begin{proof}
	The proof is  similar as to Allegretto-Xi \cite{Allegr}, where it has been obtained for differentiable functions $u, \phi\geq 0$. Indeed, expanding $R(\phi,u+\epsilon)$ according to \eqref{eq:Picone} we obtain $L(\phi,u+\epsilon)=R(\phi,u+\epsilon)$. By Young's inequality, we have
	$$
	\int\rho \frac{\phi^{q-1}}{(u+\epsilon)^{q-1}}|\nabla u|^{q-2}|\nabla \phi| |u|\,dx \leq \frac{1}{q}\int \rho |\nabla \phi|^q\,dx+\frac{q-1}{q}\int \rho \frac{\phi^q}{(u+\epsilon)^{q}}|\nabla u|^q\,dx. 
	$$
	Using this to estimate $\int L(\phi,u+\epsilon)\,dx$ we obtained  $\int L(\phi,u+\epsilon)\,dx \geq 0$. 
\end{proof}
\begin{lemma}\label{lem:Picone}
	Assume that $\rho $  satisfies \eqref{EqC} and $\lambda \geq \lambda_1(\rho)$. Let $u,  \phi \in W^{1,q, \rho}_0\cap L^\infty$, and  $u, \phi \geq 0$ in $\Omega$ such that   they  weakly satisfy to
	\begin{align}
		- &{\rm div}(\rho  |\nabla u|^{q-2}\nabla u)= \lambda b|u|^{q-2}u,\label{eq:Spect}\\
		-&{\rm div}(\rho  |\nabla \phi|^{q-2}\nabla \phi)= \lambda_1(\rho) b |\phi|^{q-2}\phi. \label{eq:Spect2}
	\end{align}
	Then  $\lambda = \lambda_1(\rho)$. 
\end{lemma}
\begin{proof}
	By Lemma  \ref{allegH} and \eqref{eq:Spect} we obtain
	\begin{align*}
		0\leq &\int_{\Omega_0} L_\rho(\phi,u+\epsilon)\,dx=\int R_\rho(\phi,u+\epsilon)\, dx=\\
		& \int\rho|\nabla \phi|^q\,dx+\int{\rm div}(\rho  |\nabla u|^{q-2}\nabla u)\frac{\phi^q}{(u+\epsilon)^{q-1}}\,dx= \\
		&\int\rho|\nabla \phi|^q\,dx-\lambda\int b u^{q-1}\frac{\phi^q}{(u+\epsilon)^{q-1}}\,dx.
	\end{align*} 
	Observe, 
	$$
	0\leq b u^{q-1}\frac{\phi^q}{(u+\epsilon)^{q-1}}\leq b \phi^q\leq b \|\phi\|^q_\infty<+\infty~~\mbox{a.e. in }~\Omega
	$$
	since $\phi \in L^\infty$. Furthermore, by \eqref{EqC2}, $ b \|\phi\|^q_\infty\in L^{\frac{\gamma}{\gamma-q}}$, and 
	$$
	bu^{q-1}\frac{\phi^q}{(u+\epsilon)^{q-1}} \to b\phi^q ~~\mbox{a.e. in }~\Omega~\mbox{as}~\epsilon \to 0. 
	$$
	Hence, by the Lebesgue dominated convergence theorem 
	$$
	\int b u^{q-1}\frac{\phi^q}{(u+\epsilon)^{q-1}}\,dx \to \int b \phi^q\,dx~\mbox{as}~\epsilon \to 0,
	$$
	and therefore, by \eqref{eq:Spect2} we have
	$$
	0\leq \int\rho|\nabla \phi|^q\,dx-\lambda \int b \phi^q\,dx=-(\lambda- \lambda_1(\rho))\int b \phi^q\,dx.
	$$
	Since $\lambda \geq \lambda_1(\rho)$, we obtain $\lambda = \lambda_1(\rho)$.

\end{proof}
\bibliographystyle{amsplain}

\end{document}